\documentclass[11pt]{article}
\usepackage{amsfonts,amsmath,amssymb}
\usepackage{amsthm}
\usepackage{enumerate}
\usepackage{graphicx}
\usepackage[hypertex]{hyperref}
\usepackage{epic}
\usepackage{eepic}
\usepackage{tikz}
\usepackage{comment}
\usepackage{multirow}
\usepackage{fouridx}
\usepackage{here}
\usepackage{amsmath,amsfonts,amsthm,amssymb,amscd}
\usepackage[all]{xy}

\usepackage{thmtools,thm-restate}

\usepackage{float}%

\usepackage{comment}
\usepackage{enumerate}

\usepackage{bm}
\makeatletter
    
    \@addtoreset{equation}{section}
  \makeatother

\newcommand{\N }{\mathbb{N}}
\newcommand{\pr}[1]{\mathbb{P}^{#1}}

\newcommand{\GL}[1]{GL_{#1}}
\newcommand{\diag}[1]{\mathrm{diag}(#1)}
\newcommand{\orb}[1]{\mathcal{O}(#1)}

\newcommand{\vp}[1]{\varphi[#1]}

\theoremstyle{plain} 
\newtheorem{thm}{Theorem}[section] 
\newtheorem{thmalph}{Theorem}

  \newtheorem{prop}[thm]{Proposition} %
  \newtheorem{lemma}[thm]{Lemma} %
\theoremstyle{definition} 
  \newtheorem{df}[thm]{Definition}
\newtheorem*{notation}{Notation}
  \newtheorem{exmp}[thm]{Example}

  \newtheorem{remark}[thm]{Remark}%

\usepackage[truedimen,margin=30truemm]{geometry}
\newcommand{\bethm}{\begin{thm}}
\newcommand{\enthm}{\end{thm}}
\newcommand{\bedef}{\begin{df}}
\newcommand{\beprop}{\begin{prop}}
\newcommand{\enprop}{\end{prop}}
\newcommand{\belem}{\begin{lemma}}
\newcommand{\enlem}{\end{lemma}}
\newcommand{\berem}{\begin{remark}}
\newcommand{\enrem}{\end{remark}}
\newcommand{\beex}{\begin{exmp}}
\newcommand{\ene}{\end{exmp}}
\newcommand{\beeq}{\begin{equation}}
\newcommand{\eneq}{\end{equation}}
\newcommand{\bealign}{\begin{align}}
\newcommand{\bear}{\begin{array}}
\newcommand{\enar}{\end{array}}
\newcommand{\beca}{\begin{cases}}
\newcommand{\enca}{\end{cases}}
\newcommand{\been}{\begin{enumerate}}
\newcommand{\enen}{\end{enumerate}}
\allowdisplaybreaks

\title{Description of $GL_3$-orbits on the quadruple projective varieties
}
\author{Naoya SHIMAMOTO \\ Chiba Institute of Technology
}
\date{}

\begin{document}
\maketitle
\begin{abstract}
This article gives a description of the diagonal $GL_3$-orbits on the quadruple projective variety $(\mathbb P^2)^4$. 
We give explicit representatives of orbits, and describe the closure relations of orbits. 
A distinguished feature of our setting is that it is the simplest case where $\diag{\GL{n}}$ has infinitely many orbits but has an open orbit in the multiple projective space $(\mathbb P^{n-1})^m$.  
%
%
%
%
\end{abstract}

\tableofcontents

\section{Introduction}

The object of this article is an orbit decomposition of the flag variety $G/P$ under the action of  a closed subgroup $H$ of a reductive group $G$ in the setting that $H$ has an open orbit in $G/P$ but $\#(H\backslash G/P)=\infty$.

The study of the double coset $H\backslash G/P$ is motivated by the representation theory. 
For example, for a real reductive algebraic group $G$, and for an algebraically defined closed subgroup $H$, Kobayashi-Oshima \cite{ko} proved that the finiteness (resp. boundedness) of the multiplicities of irreducible admissible representations $\pi$ of $G$ occurring in the induced representations $\mathrm{Ind}_H^G(\tau)$ for finite-dimensional irreducible representations $\tau$ of $H$ is equivalent to the existence of open $H$-orbits (resp. $H_c$-orbits) on $G/P_G$ (resp. $G_c/B_G$), where $P_G$ is a minimal parabolic subgroup of $G$, and $B_G$ is a Borel subgroup of the complexification $G_c$ of $G$. 
A homogeneous space $G/H$ satisfying these equivalent conditions is called a real spherical variety (resp.\ spherical variety). 
In this case, the existence of an open $H$-orbit is known to be the finiteness of $H$-orbits \cite{br1, br2, morb}. 
Needless to say, the latter implies the former, however, the converse may fail for 
a general parabolic subgroup $P$ of $G$. 
Even in such cases, the existence of open $H$-orbits or the finiteness of $H$-orbits on the flag variety $G/P$ gives useful information on the branching laws between representations of $H$ and representations of $G$ induced from characters of $P$ \cite{k,tt}. 

Kobayashi-Oshima \cite{ko} also proved that the finiteness of the dimension of the space of symmetry breaking operators ($H$-intertwining operators from irreducible admissible representations of $G$ to those of $H$) is equivalent to the existence of open $P_H$-orbits on $G/P_G$ where $H$ is reductive and $P_H$ is its minimal parabolic subgroup. 
Furthermore, such pairs $(G,H)$ are classified by Kobayashi-Matsuki \cite{km} under the assumption that $(G,H)$ are symmetric pairs. 
Moreover, an explicit description of the double coset $P_H\backslash G/P_G\cong \mathrm{diag}(H)\backslash(G\times H)/(P_G\times P_H)$ parametrises the ``families'' of symmetry breaking operators as in the works \cite{kl,ks,ks2} for the pair $(G,H)=(O(p+1,q),O(p,q))$ which is in the classification given by \cite{km}. 
Our research comes from these motivations. 

\bigskip
So far, an explicit description of the double coset $H\backslash G/P$ has been studied only when $H$ has finitely many orbits in $G/P$.  
For instance, the descriptions of $H$-orbits on $G/P$ for symmetric pairs $(G,H)$ have been well studied by Matsuki \cite{mp,mpp}, 
generalising the Bruhat decomposition where $\mathrm{diag}(G)\backslash (G\times G)/(P_1\times P_2)$ may be identified with $P_1\backslash G/P_2$. 
Matsuki \cite{mBex,mB} gave also a description of diagonal action of orthogonal groups on multiple flag varieties $G^m/(P_1\times P_2\times\cdots\times  P_m)$ under the assumption that the number of orbits is finite, referred to as ``finite type'', which were also classified in that work. 
Note that the pair $(G^m, \mathrm{diag}(G))$ is no more a symmetric pair if $m\geq 3$, and multiple flag varieties of finite type were earlier classified by \cite{mwzA,mwzC} where $G$ is of type $A$ or $C$. 



\bigskip
In this article, we are interested in the case where $\#(H\backslash G/P)=\infty$. We fix $\mathbb K$ to be an algebraically closed field with characteristic $0$, and $Q_n$ to be the maximal parabolic subgroup of $\GL{n}$ such that $GL_n/Q_n\cong\pr{n-1}$ over $\mathbb K$. 
We at first prove the following. 
\begin{thmalph}[see Theorems \ref{thm:open} and \ref{thm:finite}] \label{thm:openandfinite}
Let $n,m$ be positive integers with $n\geq 2$, 
and $\mathbb K$ be an algebraically closed field with characteristic $0$, then the followings hold: 
\begin{enumerate}
\item the number of $\diag{GL_n}$-orbits on $(\pr{n-1})^m$ is finite if and only if $m\leq 3$; 
\item there exists an open $\diag{GL_n}$-orbit on $(\pr{n-1})^m$ if and only if $m\leq n+1$. 
\end{enumerate} 
\end{thmalph}
In particular, there exist infinitely many orbits and an open orbit simultaneously if and only if $4\leq m\leq n+1$. Hence, the case $(n,m)=(3,4)$ can be regarded as the simplest case among them. 
In this article, we highlight this case, and give an explicit description of $\diag{\GL{3}}$-orbits on $(GL_3/Q_3)^4\cong(\pr{2})^4$ and their closure relations. 
For this aim, we introduce the following $\diag{GL_3}$-invariant map with a finite image: 
\beeq \label{eq:prbintro}
\begin{array}{cccc}
\pi\colon & (\mathbb P^2)^4 & \to & Map(2^{\{1,2,3,4\}},\mathbb N) \\
&\rotatebox{90}{$\in$} && \rotatebox{90}{$\in$} \\
&([v_i])_{i=1}^4 &\mapsto &\left(I\mapsto \dim\mathrm{Span}\{v_i\}_{i\in I}\right)
\end{array}
\eneq
Our strategy is to determine the image of $\pi$ in $Map(2^{\{1,2,3,4\}},\mathbb N)$, and to describe $\diag{GL_3}$-orbit decomposition of each fibre $\pi^{-1}(\varphi)$ for $\varphi\in\mathrm{Image}(\pi)$. Here is a description of $\mathrm{diag}(GL_3)$-orbits on $(GL_3/Q_3)^4\cong (\pr{2})^4$.  
\begin{thmalph}[see Theorems \ref{thm:single}, \ref{thm:parameter}, and Propositions \ref{thm:clsingle} and \ref{thm:clparameter}] \label{thm:orb34} 

\been
\item The map $\pi\colon (\mathbb P^2)^4\to Map(2^{\{1,2,3,4\}},\mathbb N)$ is invariant under the diagonal action of $G=GL_3$.  
\item
A map $\varphi\in Map(2^{\{1,2,3,4\}},\mathbb N)$ belongs to the image of $\pi$ if and only if $\varphi[\ast]$ appears in Figure \ref{hasseintro} via the correspondence $\varphi\leftrightarrow\varphi[\ast]$ (see Definitions \ref{df:map} and \ref{df:mapsplit}). 

\item For each $\varphi[\ast]$ listed in Figure \ref{hasseintro}, the fibre $\pi^{-1}(\varphi[\ast])$ is a single $\diag{G}$-orbit unless $\varphi[\ast]=\vp{6}$.The fibre $\pi^{-1}(\vp{6})$ is decomposed into infinitely many $\diag{G}$-orbits by means of $5$-dimensional orbits $\orb{5;p}$ with parameters $p\in \mathbb P^1\setminus\{0,1,\infty\}$, see Definition \ref{df:mapsplit}: 
\item If we fix $p\in\pr{1}\setminus\{0,1,\infty\}$, then the closure relations among all fibres of $\vp{\ast}$ and the orbit $\orb{5;p}$ are given by the following Hasse diagram in Figure \ref{hasseintro}. 
\enen
\begin{figure}[H]
\begin{minipage}{0.78\hsize}
\centering
\begin{tikzpicture}[scale=1.4]
\node(o8)at(-0.5,3){\scalebox{0.8}{$\vp{8}$}};
\node(o71)at(-2.6,2){\scalebox{0.8}{$\vp{7;1}$}};
\node(o72)at(-1.2,2){\scalebox{0.8}{$\vp{7;2}$}};
\node(o73)at(0.2,2){\scalebox{0.8}{${\vp{7;3}}$}};
\node(o74)at(1.6,2){\scalebox{0.8}{${\vp{7;4}}$}};
\node(o634)at(-3.5,1){\scalebox{0.8}{${\vp{6;3,4}}$}};
\node(o624)at(-2.3,1){\scalebox{0.8}{${\vp{6;2,4}}$}};
\node(o614)at(-1.1,1){\scalebox{0.8}{${\vp{6;1,4}}$}};
\node(o623)at(0.1,1){\scalebox{0.8}{${\vp{6;2,3}}$}};
\node(o613)at(1.3,1){\scalebox{0.8}{${\vp{6;1,3}}$}};
\node(o612)at(2.5,1){\scalebox{0.8}{${\vp{6;1,2}}$}};
\node(p6)at(3.8,0.7){\scalebox{0.8}{${\vp{6}}$}};
\node(o534)at(-3.5,0){\scalebox{0.8}{${\vp{5;3,4}}$}};
\node(o524)at(-2.3,0){\scalebox{0.8}{${\vp{5;2,4}}$}};
\node(o514)at(-1.1,0){\scalebox{0.8}{${\vp{5;1,4}}$}};
\node(o523)at(0.1,0){\scalebox{0.8}{${\vp{5;2,3}}$}};
\node(o513)at(1.3,0){\scalebox{0.8}{${\vp{5;1,3}}$}};
\node(o512)at(2.5,0){\scalebox{0.8}{${\vp{5;1,2}}$}};
\node(o5)at(3.8,-0.2){\scalebox{0.8}{$\orb{5;p}$}};
\node(o412)at(-2.5,-1){\scalebox{0.8}{${\vp{4;1,2}}$}};
\node(o414)at(-0.5,-1){\scalebox{0.8}{${\vp{4;1,4}}$}};
\node(o413)at(1.5,-1){\scalebox{0.8}{${\vp{4;1,3}}$}};
\node(o41)at(-3.5,-1.2){\scalebox{0.8}{${\vp{4;1}}$}};
\node(o42)at(-1.5,-1.2){\scalebox{0.8}{${\vp{4;2}}$}};
\node(o43)at(0.5,-1.2){\scalebox{0.8}{${\vp{4;3}}$}};
\node(o44)at(2.5,-1.2){\scalebox{0.8}{${\vp{4;4}}$}};
\node(o2)at(-0.5,-2.2){\scalebox{0.8}{${\vp{2}}$}};
\draw[ultra thin](-0.5,2.85)--(-2.6,2.15);
\draw[ultra thin](-0.5,2.85)--(-1.2,2.15);
\draw[ultra thin](-0.5,2.85)--(0.2,2.15);
\draw[ultra thin](-0.5,2.85)--(1.6,2.15);
\draw[ultra thin](-2.6,1.85)--(-3.5,1.15);
\draw[ultra thin](-2.6,1.85)--(-2.3,1.15);
\draw[ultra thin](-2.6,1.85)--(0.1,1.15);
\draw[ultra thin](-1.2,1.85)--(-3.5,1.15);
\draw[ultra thin](-1.2,1.85)--(-1.1,1.15);
\draw[ultra thin](-1.2,1.85)--(1.3,1.15);
\draw[ultra thin](0.2,1.85)--(-2.3,1.15);
\draw[ultra thin](0.2,1.85)--(-1.1,1.15);
\draw[ultra thin](0.2,1.85)--(2.5,1.15);
\draw[ultra thin](1.6,1.85)--(0.1,1.15);
\draw[ultra thin](1.6,1.85)--(1.3,1.15);
\draw[ultra thin](1.6,1.85)--(2.5,1.15);
\draw[ultra thin](-2.6,1.85)--(3.8,1.05);
\draw[ultra thin](-1.2,1.85)--(3.8,1.05);
\draw[ultra thin](0.2,1.85)--(3.8,1.05);
\draw[ultra thin](1.6,1.85)--(3.8,1.05);
\draw[ultra thin](3.8,1.05)--(3.8,0.85);
\draw[ultra thin](-3.5,0.85)--(-3.5,0.15);
\draw[ultra thin](-2.3,0.85)--(-2.3,0.15);
\draw[ultra thin](-1.1,0.85)--(-1.1,0.15);
\draw[ultra thin](0.1,0.85)--(0.1,0.15);
\draw[ultra thin](1.3,0.85)--(1.3,0.15);
\draw[ultra thin](2.5,0.85)--(2.5,0.15);
\draw[ultra thin](3.8,0.55)--(o5);
\draw[ultra thin](3.8,0.55)--(-3.5,0.15);
\draw[ultra thin](3.8,0.55)--(-2.3,0.15);
\draw[ultra thin](3.8,0.55)--(-1.1,0.15);
\draw[ultra thin](3.8,0.55)--(0.1,0.15);
\draw[ultra thin](3.8,0.55)--(1.3,0.15);
\draw[ultra thin](3.8,0.55)--(2.5,0.15);
\draw[ultra thin](-3.5,-0.15)--(-2.5,-0.85);
\draw[ultra thin](-2.3,-0.15)--(-0.5,-0.85);
\draw[ultra thin](-1.1,-0.15)--(1.5,-0.85);
\draw[ultra thin](0.1,-0.15)--(1.5,-0.85);
\draw[ultra thin](1.3,-0.15)--(-0.5,-0.85);
\draw[ultra thin](2.5,-0.15)--(-2.5,-0.85);
\draw[ultra thin](-3.5,-0.15)--(-3.5,-1.05);
\draw[ultra thin](-3.5,-0.15)--(-1.5,-1.05);
\draw[ultra thin](-2.3,-0.15)--(-3.5,-1.05);
\draw[ultra thin](-2.3,-0.15)--(0.5,-1.05);
\draw[ultra thin](-1.1,-0.15)--(-1.5,-1.05);
\draw[ultra thin](-1.1,-0.15)--(0.5,-1.05);
\draw[ultra thin](0.1,-0.15)--(-3.5,-1.05);
\draw[ultra thin](0.1,-0.15)--(2.5,-1.05);
\draw[ultra thin](1.3,-0.15)--(-1.5,-1.05);
\draw[ultra thin](1.3,-0.15)--(2.5,-1.05);
\draw[ultra thin](2.5,-0.15)--(0.5,-1.05);
\draw[ultra thin](2.5,-0.15)--(2.5,-1.05);
\draw[ultra thin](o5)--(3.6,-0.54);
\draw[ultra thin](3.6,-0.54)--(-2.3,-0.54);
\draw[ultra thin](-2.3,-0.54)--(-3.5,-1.05);
\draw[ultra thin](-0.75,-0.54)--(-1.5,-1.05);
\draw[ultra thin](0.7,-0.54)--(0.5,-1.05);
\draw[ultra thin](2.6,-0.54)--(2.5,-1.05);
\draw[ultra thin](-2.5,-1.15)--(-0.5,-2.05);
\draw[ultra thin](-0.5,-1.15)--(-0.5,-2.05);
\draw[ultra thin](1.5,-1.15)--(-0.5,-2.05);
\draw[ultra thin](-3.5,-1.35)--(-0.5,-2.05);
\draw[ultra thin](-1.5,-1.35)--(-0.5,-2.05);
\draw[ultra thin](0.5,-1.35)--(-0.5,-2.05);
\draw[ultra thin](2.5,-1.35)--(-0.5,-2.05);
\draw[ultra thin](o5)--(3.8,-1.25);
\draw[ultra thin](3.8,-1.25)--(-0.5,-2.05);
\end{tikzpicture}
\caption{Hasse diagram}
\label{hasseintro}
\end{minipage}
\hspace{-7mm}
\begin{minipage}{0.19\hsize}
\centering
\begin{tikzpicture}[scale=1.3]
\node(o8)at(0,3){\scalebox{0.8}{$\{{\vp{8}}\}$}};
\node(o7)at(0,2){\scalebox{0.8}{$\{{\vp{7;\ast}}\}$}};
\node(o6)at(-0.4,1){\scalebox{0.8}{$\{{\vp{6;\ast,\ast}}\}$}};
\node(p6)at(0.8,0.7){\scalebox{0.8}{$\{{\vp{6}}\}$}};
\node(o5)at(-0.4,0){\scalebox{0.8}{$\{{\vp{5;\ast,\ast}}\}$}};
\node(o5p)at(0.8,-0.45){\scalebox{0.8}{$\{\orb{5;q}\left|q\in O_p\right.\}$}};
\node(o4)at(-0.8,-1){\scalebox{0.8}{$\{{\vp{4;\ast,\ast}}\}$}};
\node(o4')at(0.2,-1.4){\scalebox{0.8}{$\{{\vp{4;\ast}}\}$}};
\node(o2)at(0,-2.2){\scalebox{0.8}{$\{{\vp{2}}\}$}};
\draw[ultra thin](0,2.85)--(0,2.15);
\draw[ultra thin](0,1.85)--(-0.4,1.15);
\draw[ultra thin](0,1.85)--(0.8,0.85);
\draw[ultra thin](-0.4,0.85)--(-0.4,0.15);
\draw[ultra thin](0.8,0.55)--(-0.4,0.15);
\draw[ultra thin](0.8,0.55)--(o5p);
\draw[ultra thin](-0.4,-0.15)--(-0.8,-0.8);
\draw[ultra thin](-0.4,-0.15)--(0.2,-1.25);
\draw[ultra thin](-0.8,-1.15)--(0,-2.05);
\draw[ultra thin](0,-1.55)--(0,-2.05);
\draw[ultra thin](0.9,-0.6)--(0.2,-1.25);
\draw[ultra thin](0.9,-0.6)--(0.9,-1.55);
\draw[ultra thin](0.9,-1.55)--(0,-2.05);
\end{tikzpicture}
\hspace{-7mm}
\caption{mod $\mathcal{S}_4$}
\label{hassemods4}
\end{minipage}
\end{figure}
\end{thmalph}

\berem
\begin{enumerate}
\item Since 
$\diag{G}\backslash(\mathbb P^2)^4$ admits a natural action of the symmetric group $\mathcal{S}_4$, one can consider the quotient $\left(\mathrm{diag}(G)\backslash (\mathbb P^2)^4\right)/\mathcal S_4$ in Figure \ref{hassemods4} where $O_p:=\{p,\frac1p,1-p,\frac1{1-p},1-\frac1p,\frac1{1-\frac1p}\}$. 
\item The symbol $\varphi[\ast]\in Map(2^{\{1,2,3,4\}},\mathbb N)$ has the following information: the dimension of the fibre $\pi^{-1}\left(\varphi[k;J]\right)\subset(\mathbb P^2)^4$ is $k$, where $J\subset \{1,2,3,4\}$ is another parameter. We observe that $\varphi[k;J]$ occurs in Figure \ref{hasseintro} if and only if so is $\varphi[k;\sigma J]$ for $\sigma\in\mathcal S_4$. 
\end{enumerate}
\enrem

\begin{notation}
We set $\N=\{0,1,2,3,\ldots\}$, and $[m]$ denotes the set $\{1,2,\ldots,m \}\subset \mathbb N$ for $m\in \mathbb N$. 

We let $\mathbb K$ be an algebraically closed field with characteristic $0$. For a vector $v\in \mathbb K^n\setminus\{\bm 0\}$, we write the $\mathbb K^\times$-orbit through $v$ as $[v]\in\pr{n-1}$. Similarly for a matrix $(v_i)_{i=1}^m\in M(n,m)$ without any columns equal to $\bm 0$, the notion $[v_i]_{i=1}^m$ denotes the $m$-tuple of $\mathbb P^{n-1}$. For an $m$-tuple $[v_i]_{i=1}^m$ in $\mathbb P^{n-1}$, we write the subspace spanned by these $m$ vectors by $\langle v_i\rangle_{i=1}^m$. Furthermore, $\{e_i\}_{i=1}^n$ denotes the standard basis of $\mathbb K^n$. 

\end{notation}

\section{Existence of open orbits and finiteness of orbits}
\label{openandfinite}

In this section, we prove Theorem \ref{thm:openandfinite}, which determines the existence of open orbits and finiteness of orbits under the diagonal action of $\GL{n}$ on the multiple projective space 
$(\pr{n-1})^m$. 

\begin{thm} \label{thm:open}
Let $n\geq 2$ and $m$ be positive integers. 
There exists an open $\diag{\GL{n}}$-orbit on 
$(\pr{n-1})^m$ if and only if $n\geq m-1$.
\end{thm}
\begin{proof}
\been
\item 
If $n\geq m$, then we can take an element $[e_i]_{i=1}^m\in (\mathbb P^{n-1})^m$. 
Its stabiliser is 
\[
\left\{\left.\left(\bear{cc|c}  \multicolumn{2}{c|}{\multirow{2}{*}{\scalebox{1.5}{$C$}}} & \multirow{2}{*}{\scalebox{1.2}{$\ast$}} \\ && \\  \hline \multicolumn{2}{c|}{\scalebox{1.2}{$0$}} & \scalebox{1.2}{$\ast$}  \enar\right)\right| \;C=\diag{c_1,c_2,\ldots,c_m}\in\GL{m} \;\right\}\subset\GL{n},  
\]
and $\diag{GL_n}\cdot[e_i]_{i=1}^m$ is the open orbit since
$
\dim\left(\diag{GL_n}\cdot[e_i]_{i=1}^m\right)=\dim\GL{n}-(m+n(n-m))=(n-1)m=\dim (\mathbb P^{n-1})^m. 
$
\item 
If $n+1=m$, then we can take a element $v=[v_i]_{i=1}^{n+1}\in (\mathbb P^{n-1})^{n+1}$ where $v_i:=e_i$ for $1\leq i\leq n$ and $v_{n+1}:=\sum_{k=1}^ne_k$.
If $g\in \GL{n}$ stabilises $[v_i]=[e_i]$ for $1\leq i\leq n$, then it is a diagonal matrix, and if it stabilises $[v_{n+1}]= [\sum_{k=1}^ne_k]$, then it is actually a scalar matrix. 
Hence $\diag{GL_n}\cdot v$ is the open orbit since  
$
\dim\left(\diag{GL_n}\cdot v\right)=\dim\GL{n}-1=(n-1)(n+1)=\dim (\mathbb P^{n-1})^{n+1}. 
$
\item 
Let $n+2\leq m$. 
Remark that the centre of $GL_n$ acts trivially on $\mathbb P^{n-1}$, and the dimension of any $\diag{GL_n}$-orbit in $(\mathbb P^{n-1})^m$ is less than $n^2$. Hence there is no open orbit since 
$
\dim (\mathbb P^{n-1})^m=(n-1)m>(n-1)(n+1)=n^2-1. 
$
\enen
\end{proof}
\begin{thm} \label{thm:finite} 
Let $n\geq 2$ and $m$ be positive integers. 
There are only finitely many $\diag{\GL{n}}$-orbits on 
$(\pr{n-1})^m$ if and only if $m\leq 3$. 
\end{thm}
\begin{proof}
\been
\item  
The action of $GL_n$ on $\mathbb P^{n-1}$ is transitive, hence there is only one orbit. 
\item
From Bruhat decomposition, we have $\diag{GL_n}\backslash (\mathbb P^{n-1})^2=\diag{GL_n}\backslash (GL_n/Q_n)^2\cong Q_n\backslash GL_n/Q_n$ where $Q_n$ is the maximal parabolic subgroup of $GL_n$ with the Levi subgroup $GL_1\times GL_{n-1}$. It has a one to one correspondence with $\mathcal S_{n-1}\backslash \mathcal S_n/\mathcal S_{n-1}$, which is a two point set. 
\item 
If $m=3$, we have
\bealign
\diag{GL_n}\backslash (\mathbb P^{n-1})^3&= \diag{GL_n}\backslash\left\{[v_1, v_2, v_3]\in (\mathbb P^{n-1})^3\left|[v_1]\neq[v_2]\neq[v_3]\neq[v_1]\right.\right\}  \notag \\
&\;\;\;\;\amalg\bigcup_{1\leq i<j\leq 3}\diag{GL_n}\backslash\left\{[v_1,v_2,v_3]\in (\mathbb P^{n-1})^3\left|[v_i]=[v_j]\right.\right\}.  \label{eq:xn3decomp}
\end{align}
For $1\leq i< j\leq3$, we have 
\beeq \label{eq:xn3decomp1}
\diag{GL_n}\backslash\left\{[v_1,v_2,v_3]\in (\mathbb P^{n-1})^3\left|[v_i]=[v_j]\right.\right\}\cong \diag{GL_n}\backslash (\mathbb P^{n-1})^2, 
\eneq
hence it is finite from the previous case. 
Furthermore, we have
\bealign
&\left\{[v_1,v_2,v_3]\in (\mathbb P^{n-1})^3\left|[v_1]\neq[v_2]\neq[v_3]\neq[v_1]\right.\right\} \notag \\
&\;\;=\diag{GL_n}\cdot [e_1,e_2,e_1+e_2]\ \ \bigl(\ \amalg\ \diag{GL_n}\cdot [e_1,e_2,e_3]\bigr). \label{eq:xn3decomp2}
\end{align}
Remark that the inside of the bracket occurs only if $n\geq 3$. 
Combining (\ref{eq:xn3decomp}), (\ref{eq:xn3decomp1}), and (\ref{eq:xn3decomp2}), $\diag{GL_n}\backslash (\mathbb P^{n-1})^3$ is finite. 
\item 
If $m\geq 4$, then we can take $v(p):=[e_1,e_2,e_1+e_2, p ,p, \ldots,p]\in (\mathbb P^{n-1})^m$ for $p\in\pr{1}$. 
For $p,q\in\pr{1}$, assume that $g\in\GL{n}$ satisfies $\diag{g}\cdot v(p)=v(q)$. 
Since $g$ fixes $[e_1], [e_2]$, and $[e_1+e_2]$, the restricted linear map $g|_{\langle e_1,e_2\rangle}$ is a scalar action, hence $q=g\cdot p=p$, and $v(p)=v(q)$. 
Hence $\{\diag{GL_n}\cdot v(p)\}_{p\in\pr{1}}$ is a distinct family of orbits with infinitely many elements. 
\enen
\end{proof}

Combining these results, we completed the proof of Theorem \ref{thm:openandfinite}.

\section{General structures of orbits on multiple projective spaces}
\label{general}

In Section \ref{openandfinite}, we observed the existence of open orbits and finiteness of orbits on 
$(\mathbb P^{n-1})^m$ under the diagonal action of $\GL{n}$. 
In this section, we consider further general properties of diagonal $GL_n$-orbits on 
$(\pr{n-1})^m$.


\subsection{Indecomposable splitting of configuration spaces}

To simplify the description of ${GL_n}$-orbit decomposition of $(\mathbb P^{n-1})^m$, we introduce some ${GL_n}$-invariant subsets of $(\mathbb P^{n-1})^m$ and consider the decompositions into them. 

For elements of $(\mathbb P^{n-1})^m$, we can define the notion of indecomposability as follows: 
\begin{df} \label{df:indecomposable} 
For $v=[v_i]_{i=1}^m\in (\mathbb P^{n-1})^m$, we say $v$ is decomposable if there exists a pair $\emptyset\neq I,J\subset[m]$ such that $I\amalg J=[m]$ and $\langle v_i\rangle_{i\in I}\cap \langle v_j\rangle_{j\in J}=\bm 0$. Conversely, if $v$ does not have such decomposition, we say it to be indecomposable. 
\end{df} 

It is equivalent to the notion of indecomposability as objects in the flag category, and with this notion, we can consider indecomposable splittings of elements in $(\mathbb P^{n-1})^m$ as follows: 
\begin{df}\label{df:splitting} \begin{enumerate}
\item Define a set $\mathcal P_{m}$ as bellow:
\[\mathcal P_{m}:=\left\{\{(I_k,r_k)\}_{k=1}^l\left|\ \coprod_{k=1}^lI_k=[m],\ I_k\neq\emptyset,\ r_k\in\mathbb N\right.\right\}.\]
\item Define the map $\varpi\colon (\mathbb P^{n-1})^m\to \mathcal P_{m},\  v\mapsto \{(I_k,r_k)\}_{k=1}^l$ where 
\begin{enumerate}
\item $\langle v_i\rangle_{i=1}^m=\bigoplus_{k=1}^l \langle v_i\rangle_{i\in I_k}$ and  $r_k=\dim\langle v_i\rangle_{i\in I_k}$, 
\item $[v_i]_{i\in I_k}\in (\mathbb P^{n-1})^{\#I_k}$ is indecomposable. 
\end{enumerate}
\end{enumerate}
\end{df}
The map $\varpi$ is well-defined since it is generally known that indecomposable splittings of flags are unique up to isomorphisms. However we can also check it elementarily in this case. Let $\{(I_k,r_k)\}_{k=1}^l$ and $\{(J_k,s_k)\}_{k=1}^{l'}\in\mathcal P_m$ satisfy the two conditions i) and ii) in Definition \ref{df:splitting} for $[v_i]_{i=1}^m\in (\mathbb P^{n-1})^m$. If $I_k\cap J_{k'}\neq\emptyset$ and $I_k\setminus J_{k'}\neq\emptyset$, then $\langle v_i\rangle_{i\in I_k\cap J_{k'}}\cap\langle v_i\rangle_{i\in I_k\setminus J_{k'}}=\bm 0$ from i), and it contradicts to the indecomposability of $[v_i]_{i\in I_k}$ assumed in ii). Hence either $I_k\cap J_{k'}=\emptyset$ or $I_k=J_{k'}$ holds. 
\begin{exmp} For $(n,m)=(3,4)$, we have 
\begin{align*}
\varpi([e_1,e_1,e_2,e_3])&=\left\{(\{1,2\},1), (\{3\},1),(\{4\}, 1)\right\}, \\
\varpi([e_1,e_2,e_3,e_2+e_3])&=\left\{(\{1\},1),(\{2,3,4\},2)\right\}.
\end{align*}
\end{exmp}
On the other hand, in this article, we observe mainly the following map: 
\begin{equation} \label{eq:rankmatrix}
\begin{array}{cccc}
\pi\colon & (\mathbb P^{n-1})^m & \longrightarrow & Map(2^{[m]},\mathbb N) \\
& \rotatebox{90}{$\in$} &&\rotatebox{90}{$\in$} \\
& [v_i]_{i=1}^m &\mapsto & \left(I\mapsto \dim\langle v_i\rangle_{i\in I}\right)
\end{array}
\end{equation}
We can characterise the correspondence between $\varpi$ and $\pi$ as bellow: 
\begin{lemma} \label{thm:rho} For the maps $\varpi$ and $\pi$ defined in Definition \ref{df:splitting} and (\ref{eq:rankmatrix}), there exist following maps satisfying the diagram bellow. 
\[\xymatrix{
& Map(2^{[m]},\mathbb N) \ar@{.>>}[r]^-{\rho} &\mathcal P_{m} \\
GL_n\backslash (\mathbb P^{n-1})^m\ar@{.>>}[r]^-{\tilde\pi}& \mathrm{Image}(\pi) \ar@<-0.3ex>@{^{(}->}[u] \ar@{.>>}[r]^-{\left.\rho\right|} & \mathrm{Image}(\varpi) \ar@<-0.3ex>@{^{(}->}[u]\\
(\mathbb P^{n-1})^m\ar@{->>}[u]\ar@{->>}[ru]_-{\pi} \ar@{->>}[rru]_-{\varpi}
}\]
\end{lemma}
\begin{proof}
Since the map $\pi$ is $GL_n$-invariant, we can define the map $\tilde\pi$ clearly. 

Now for a map $\varphi\colon 2^{[m]}\to\mathbb N$, we define it to be decomposable if there exists a partition $I_1\amalg I_2=[m]$, $I_1,I_2\neq\emptyset$ which satisfies $\varphi(I)=\varphi(I\cap I_1)+\varphi(I\cap I_2)$ for all $I\subset [m]$. Then we can define the map $\rho\colon Map(2^{[m]},\mathbb N)\to \mathcal P_m$, $\varphi\mapsto \{(I_k,r_k)\}_{k=1}^m$ where
\begin{enumerate}
\item $\varphi(I)=\sum_{k=1}^l \varphi(I\cap I_k)$ for $I\subset [m]$, and $r_k=\varphi(I_k)$, 
\item the restricted map $\left.\varphi\right|_{I_k}$ is indecomposable for each $1\leq k\leq l$. 
\end{enumerate}
The map $\rho$ is well-defined. Indeed, let $\{(I_k,r_k)\}_{k=1}^l$ and $\{(J_k,s_k)\}_{k=1}^{l'}\in\mathcal P_m$ satisfy the conditions (1) and (2) for $\varphi\colon 2^{[m]}\to\mathbb N$. If $I_k\cap J_{k'}\neq\emptyset$ and $I_k\setminus J_{k'}\neq\emptyset$, then for $I\subset I_k$, 
\begin{align*}
\varphi(I)&=\sum_{k''=1}^{l'}\varphi(J_{k''}\cap I)=\varphi(J_{k'}\cap I)+\sum_{k''=1}^{l'}\varphi(J_{k''}\cap I\setminus J_{k'})=\varphi(J_{k'}\cap I)+\varphi(I\setminus J_{k'}) \\
&= \varphi(I_k\cap J_{k'}\cap I)+\varphi(I_k\setminus J_{k'}\cap I)\end{align*}
from (1), and it contradicts to the indecomposability of $\left.\varphi\right|_{I_k}$ in (2). Hence either $I_k\cap J_{k'}=\emptyset$ or $I_k=J_{k'}$ holds, which concludes the well-definedness of $\rho$. 

Now for $\varpi([v_i]_{i=1}^m)=\{(I_k,r_k)\}_{k=1}^m$, we have
\begin{enumerate}
\item From $\langle v_i\rangle_{i=1}^m=\bigoplus_{k=1}^l \langle v_i\rangle_{i\in I_k}$, we have $\dim\langle v_i\rangle_{i\in I}=\sum_{i=1}^l \langle v_i\rangle_{i\in I_k\cap I}$. 
\item If $\emptyset\neq J,J'$ and $J\amalg J'=I_k$, then from the indecomposability of $[v_i]_{i\in I_k}$, we have $\langle v_i\rangle_{i\in I_k\cap J}\cap \langle v_i\rangle_{i\in I_k\cap J'}\neq \bm 0$. Hence $\dim\langle v_i\rangle_{i\in I_k}< \dim\langle v_i\rangle_{i\in I_k\cap J}+\dim \langle v_i\rangle_{i\in I_k\cap J'}$. 
\end{enumerate}
Hence we have shown that $\rho\circ\pi=\varpi$. 
\end{proof}

\subsection{Some typical decompositions of multiple projective spaces
}
\label{typical}

For the map $\varpi$ defined in Definition \ref{df:splitting}, there are some remarkable properties. 
\begin{lemma} \label{thm:imagevarpi} For the map $\varpi\colon (\mathbb P^{n-1})^m\to\mathcal P_{m}$ defined in Definition \ref{df:splitting}, an element $\left\{(I_k,r_k)\right\}_{k=1}^l$ of $\mathcal P_{m}$ is contained in the image of $\varpi$ if and only if the followings are satisfied: 
\begin{enumerate}
\item $r_k=1$, or $2\leq r_k\leq \#I_k-1$ for each $1\leq k\leq l$, 
\item $\sum_{k=1}^lr_k\leq  n$.
\end{enumerate}
\end{lemma}
\begin{proof}
If $\varpi([v_i]_{i=1}^m)=\left\{(I_k,r_k)\right\}_{k=1}^l$, since $\langle v_i\rangle_{i=1}^m=\bigoplus_{k=1}^l\langle v_i\rangle_{i\in I_k}$ and $r_k=\dim\langle v_i\rangle_{i\in I_k}$, the second condition holds obviously. 
Furthermore, let us assume that the first condition fails, in other words, assume that $2,\#I_k\leq r_k$ is satisfied for some $k$. Since $\#I_k<r_k=\dim\langle v_i\rangle_{i\in I_k}$ can not occur, the equiality must hold. Hence $\{v_i\}_{i\in I_k}$ is linearly independent, which contradicts to the indecomposability. 

Conversely, if the two conditions are satisfied for $\{(I_k,r_k)\}_{k=1}^l\in\mathcal P_m$, we define as follows:
\begin{enumerate}
\item $R(k):=\sum_{k'\leq k}r_{k'}$ for $0\leq k\leq l$. In particular, $0=R(0)<R(1)<R(2)<\cdots R(l)\leq n$. 
\item For $I_k=\{i(1)<i(2)<\cdots<i({\#I_k})\}$, set 
\[v_{i(j)}:=\begin{cases} e_{R(k-1)+j} & 1\leq j\leq r_k, \\  \sum_{j'=1}^{r_k}e_{R(k-1)+j'} & r_k+1\leq j  \leq \#I_k. \end{cases}\]
\end{enumerate}
Then we have 
$\langle v_i\rangle_{i=1}^m=\langle e_j\rangle_{j=1}^{R(l)}=\bigoplus_{k=1}^l \langle e_{R(k-1)+j}\rangle _{j=1}^{r_k}=\bigoplus_{k=1}^l\langle v_i\rangle_{i\in I_k}$, 
and each $\langle v_i\rangle_{i\in I_k}$ is indecomposable since the first $r_k+1$ vectors are in a general position in the $r_k$-dimensional space. Hence $\varpi([v_i]_{i=1}^m)=\{(I_k,r_k)\}_{k=1}^l$. 
\end{proof}

\begin{lemma} \label{thm:rhobij} If we define the subset of $\mathcal P_m$ by 
\[\mathcal P_{n,m}':=\left\{\{(I_k,r_k)\}_{k=1}^m\in\mathrm{Image}(\varpi)\subset\mathcal P_m\left|\ r_k=1,\ \textrm{or}\ 2\leq r_k=\#I_k-1\ (1\leq k\leq l)\right.\right\}, \]
then for each element $\{(I_k,r_k)\}_{k=1}^l\in\mathcal P_{n,m}'$, the fibre $\varpi^{-1}(\left\{(I_k,r_k)\right\}_{k=1}^l)\subset (\mathbb P^{n-1})^m$ is a single $GL_n$-orbit. Furthermore, this orbit coincides with the fibre $\pi^{-1}(\varphi)$ where 
\[\varphi\colon 2^{[m]}\to\mathbb N\colon I\mapsto\sum_{k=1}^l \min\left\{ \#(I_k\cap I), r_k\right\}. \]
In particular, we have the following bijections. 
\[\xymatrix{
GL_n\backslash (\mathbb P^{n-1})^m \ar@{>>}@/^20pt/[rr]^-{\tilde\varpi}\ar@{>>}[r]^-{\tilde\pi}& \mathrm{Image}(\pi)  \ar@{>>}[r]^-{\left.\rho\right|} & \mathrm{Image}(\varpi) \\
GL_n\backslash\varpi^{-1}(\mathcal P'_{n,m}) \ar[r]^-{\simeq}\ar@<-0.3ex>@{^{(}->}[u]  & \left.\rho\right|^{-1}(\mathcal P'_{n,m}) \ar[r]^-{\simeq}\ar@<-0.3ex>@{^{(}->}[u] & \mathcal P'_{n,m} \ar@<-0.3ex>@{^{(}->}[u]
}\]
\end{lemma}

\begin{proof}
We define $v_1,\ldots,v_m$ as in Lemma \ref{thm:imagevarpi} for $\{(I_k,r_k)\}_{k=1}^l\in\mathcal P_{n,m}'$. Now let us consider an element $[w_i]_{i=1}^m\in\varpi^{-1}(\{(I_k,r_k)\}_{k=1}^l)$. 
\begin{enumerate}
\item If $r_k=1$, there is an linear isomorphism $f_k$ between $1$-dimensional subspaces $\langle e_{R(k)}\rangle=\langle v_i\rangle_{i\in I_k}$ and $\langle w_i\rangle_{i\in I_k}$. In particular, $[f_k(w_i)]=[v_i]$ for all $i\in I_k$. 
\item For the case $2\leq r_k=\#I_k-1$, if $r_k$-tuple of $\{w_i\}_{i\in I_k}$ is \emph{not} linearly independent, then it spans strictly less than $r_k$-dimensional subspace of $\langle w_i\rangle_{i\in I_k}$, which does not contain the rest one vector. It contradicts to the indecomposability of $[w_i]_{i\in I_k}$. Hence all $r_k$-tuples of $\{w_i\}_{i\in I_k}$ must be linearly independent. 

From this observation, since $\{w_{i(1)}, w_{i(2)},\ldots,w_{i(r_k)}\}$ is a basis of the $r_k$-dimensional space $\langle w_i\rangle_{i\in I_k}$, there exists an linear combination $w_{i(r_k+1)}=\sum_{j=1}^{r_k}c_jw_{i(j)}$. If $c_1=0$, then $\{w_{i(j)}\}_{j=2}^{r_k+1}$ is linearly dependent, which contradicts to the observation above. Hence $c_1\neq 0$. Similarly we can prove $c_j\neq 0$ for all $1\leq j\leq r_k$. 

Now we can define a linear isomorphism $f_k$ between the $r_k$-dimensional spaces $\langle w_i\rangle_{i\in I_k}=\langle w_{i(j)}\rangle_{j=1}^{r_k}$ and $\langle v_i\rangle_{i\in I_k}=\langle e_{R(k-1)+j}\rangle_{j=1}^{r_k}$ by sending $c_jw_{i(j)}$ to $v_{i(j)}=e_{R(k-1)+j}$ for $1\leq j\leq r_k$, which leads that $f_k(w_{i(r_k+1)})=f_k\left(\sum_{j=1}^{r_k}c_jw_{i(j)}\right)=\sum_{j=1}^{r_k}e_{R(k-1)+j}=v_{i(r_k+1)}$. 
\end{enumerate}
Now, since $\langle w_i\rangle_{i=1}^m=\bigoplus_{k=1}^l\langle w_i\rangle_{i\in I_k}$ and  $\langle v_i\rangle_{i=1}^m=\bigoplus_{k=1}^l\langle v_i\rangle_{i\in I_k}$, by taking the direct product of these linear isomorphisms $f_k\colon \langle w_i\rangle_{i\in I_k}\simeq\langle v_i\rangle_{i\in I_k}$ and some linear isomorphism between the complement space, we have obtained an isomorphism $g\in GL_n$ such that $g\cdot[w_i]_{i=1}^m=[v_i]_{i=1}^m$. 

\bigskip
Now, since any $r_k$-tuple of $\{w_i\}_{i\in I_k}$ are linearly independent, $\dim \langle w_i\rangle_{i\in I_k\cap I}=\min\{\#(I_k\cap I),r_k\}$. Furthermore, since $\langle w_i\rangle_{i=1}^m=\bigoplus_{k=1}^l\langle w_i\rangle_{i\in I_k}$, we have
\[\dim\langle w_i\rangle_{i\in I}=\sum_{k=1}^l \dim\langle w_i\rangle_{i\in I_k\cap I}=\sum_{k=1}^l \min\{\#(I_k\cap I),r_k\}.\]
Hence $\varpi^{-1}(\{(I_k,r_k)\}_{k=1}^l)\subset \pi^{-1}(\varphi)$. On the hand, for an element $v\in \pi^{-1}(\rho)$, we have $\varpi(v)=\rho\circ\pi(v)=\rho(\varphi)=\{(I_k,r_k)\}_{k=1}^l$, which concludes the converse inclusion. 
\end{proof}

\section{Description of orbits}
\label{main}

From this section, we set $G$ to be the general linear group of the degree $3$ over algebraically closed field $\mathbb K$ with the characteristic $0$, and $Q$ be the maximal parabolic subgroup of $G$ such that $G/Q\cong\pr{2}$. 

In Section \ref{prbdecomp}, we determine the image of the map $\varpi$ and its subset $\mathcal P_{3,4}'$ defined in Definition \ref{df:splitting} and Lemma \ref{thm:rhobij}. Then we can obtain a description of $G$-orbit decomposition of $\varpi^{-1}(\mathcal S_{3,4})'\subset (\mathbb P^2)^4$ using $G$-invariant fibres of $\varpi$ and $\pi$ from Lemma \ref{thm:rhobij} (see Theorem \ref{thm:single}). 

Then in Section \ref{orbdecomp}, we describe the orbit decomposition of the $G$-invariant fibre of $\varpi$ on $\mathrm{Image}(\varpi)\setminus \mathcal P_{3,4}'$ to complete the description of all orbits (see Theorem \ref{thm:parameter}). 


\subsection{
Decomposition with the indecomposable splitting}
\label{prbdecomp}

First of all, we describe the image of the map $\varpi\colon (\mathbb P^2)^4\to \mathcal P_4$ defined in Definition \ref{df:splitting}: 
\begin{lemma}\label{thm:splittinglist} For $\{(I_k,r_k)\}_{k=1}^l\in \mathcal P_4$, it is contained in the image of $\varpi\colon (\mathbb P^2)^4\to\mathcal P_4$ (see Definition \ref{df:splitting}) if and only if the tuple $\{(\#I_k,r_k)\}_{k=1}^l$ is one of the followings:
\[\{(4,1)\},\ \{(3,1),(1,1)\},\{(2,1),(2,1)\},\ \{(4,2)\},\ \{(1,1),(1,1),(2,1)\},\ \{(3,2),(1,1)\},\ \{(4,3)\}.\]
\end{lemma}

\begin{proof}
From Lemma \ref{thm:imagevarpi}, the tuple $\{(I_k,r_k)\}_{k=1}^l\in \mathcal P_4$ is in the image of $\varpi\colon (\mathbb P^2)^4\to Map(2^{[4]},\mathbb N)$ if and only if $\sum_{k=1}^lr_k\leq 3$ and $r_k=1$ or $2\leq r_k\leq \#I_k-1$ for each $1\leq k\leq l$. Hence we can classify the image  of $\varpi$ as follows:
\begin{enumerate}
\item if $l=1$, then $\#I_1=4$ and the possibility of $r_1$ is equal to or less than $4-1=3$. hence the possibilities of $\{(\#I_k,r_k)\}_{k=1}^l$ are $\{(4,3)\}$, $\{(4,2)\}$, $\{(4,1)\}$. 
\item If $l=2$, the possibilities of partitions are $\{\#I_1,\#I_2\}=\{1,3\}$ or $\{2,2\}$. The first one corresponds to $\{(1,1),(3,2)\}$ or $\{(1,1),(3,1)\}$. The second one corresponds to $\{(2,1),(2,1)\}$. 
\item If $l=3$,  the possibilities of partitions are $\{\#I_1,\#I_2,\#I_2\}=\{1,1,2\}$ hence it corresponds to $\{(1,1),(1,1),(2,1)\}$.
\item If $l\geq 4$, then $\sum_{k=1}^lr_k\geq 4$ contradicts to the condition $\sum_{k=1}^lr_k\leq 3$.   
\end{enumerate}
\end{proof}

\begin{remark} From the definition of $\mathcal P_{3,4}'$ in Lemma \ref{thm:rhobij}, the tuple $\{(I_k,r_k)\}_{k=1}^l\in \mathrm{Image}(\varpi)$ is in $\mathcal P_{3,4}'$ if and only if $r_k=1$ or $2\leq r_k= \#I_k-1$ for each $1\leq k\leq l$. Hence from the classification in Lemma \ref{thm:splittinglist}, we have $\mathrm{Image}(\varpi)\setminus \mathcal P_{3,4}'=\{\{([4],2)\}\}$. 
\end{remark}

\begin{df} \label{df:map}
We define maps $\varphi[\ast]\in Map( 2^{[4]},\mathbb N)$ as in Table \ref{tab:vpsingle}, according to the correspondence between $\{(I_k,r_k)\}_{k=1}^l\in \mathcal P_{3,4}'$ classified in Lemma \ref{thm:splittinglist} and maps $2^{[4]}\to\mathbb N$, $I\mapsto \sum_{k=1}^l\min\{\#I_k,r_k\}$. In Table \ref{tab:vpsingle}, we set $\{i,j,k,l\}=[4]$. 
\begin{table}[h]
\centering
\small
\begin{tabular}{|c|c|c|c|}
\hline
$\{(\#I_k,r_k)\}_{k=1}^l$ & $\{I_k\}_{k=1}^l $ & $\left.\rho\right|^{-1}(\{(I_k,r_k)\}_{k=1}^l)$ &  representative of the orbit $\pi^{-1}(\varphi[\ast])$ \\
\hline \hline
$\{(4,1)\}$ & 
$\{\{1,2,3,4\}\}$ & $\displaystyle{\vp{2}}$ &  $[e_1,e_1,e_1,e_1]$ \\
\hline
$\{(1,1),(3,1)\}$& 
$\{\{i\},\{j,k,l\}\}$ & $\displaystyle{\vp{4;i}}$ & $[e_1,e_2,e_2,e_2]$ if $i=1$ \\
\hline
$\{(2,1),(2,1)\}$&
$\{\{i,j\},\{k,l\}\}$ & $\displaystyle{\vp{4;i,j}}$ & $[e_1,e_1,e_2,e_2]$ if $\{i,j\}=\{1,2\}$ \\
\hline
$\{(1,1),(1,1),(2,1)\}$&
$\{\{i\},\{j\},\{k,l\}\}$ & $\displaystyle{\vp{6;k,l}}$ & $[e_1,e_2,e_3,e_3]$ if $\{k,l\}=\{3,4\}$ \\
\hline
$\{(1,1),(3,2)\}$&
$\{\{i\},\{j,k,l\}\}$& $\displaystyle{\vp{7;i}}$ & $[e_1,e_2,e_3,e_2+e_3]$ if $i=1$ \\
\hline
$\{(4,3)\}$&
$\{\{1,2,3,4\}\}$ & $\displaystyle{\vp{8}}$ & $[e_1,e_2,e_3,e_1+e_2+e_3]$ \\
\hline
\end{tabular}
\caption{Definition of maps $\varphi[\ast]$}
\label{tab:vpsingle}
\end{table}
\end{df}
\begin{exmp}
For instance, the map $\varphi[7;4]$ corresponds to the partition $[4]=\{4\}\amalg [3]$, and we have $\varphi[7;i](I)=\min\{\#(I\cap\{4\}),1\}+\min\{\#(I\setminus\{4\}),2\}$. In other words, $[v_i]_{i=1}^4$ is in the fibre $\pi^{-1}\left(\varphi[7;1]\right)$ if and only if $\{v_2,v_2,v_3\}$ is in a general position in a $2$-dimensional space and $v_1$ is transversal to it. 
\end{exmp}
\berem
We can see that $\vp{4;i,j}=\vp{4;j,i}=\vp{4;k,l}=\vp{4;l,k}$, $\vp{5;i,j}=\vp{5;j,i},\;\vp{6;i,j}=\vp{6;j,i}$ where $\{1,2,3,4\}=\{i,j,k,l\}$. 
\enrem

From Lemma \ref{thm:rhobij}, remark that for all $\left.\rho\right|^{-1}(\{(I_k,r_k)\}_{k=1}^l)=\varphi[\ast]\in Map(2^{[4]},\mathbb N)$ defined in Table \ref{tab:vpsingle}, the fibre $\pi^{-1}(\varphi[\ast])\subset (\mathbb P^2)^4$ is a single $G$-orbit through the element $v\in \varpi^{-1}(\{(I_k,r_k)\}_{k=1}^l)$ defined in Lemma \ref{thm:imagevarpi} as listed in Table \ref{tab:vpsingle}. Furthermore, the dimension of the orbit through it is $\dim G-\dim \mathrm{Stab}_v=3\sum_{k=1}^lr_k-l$ since an isomorphism $g\in G$ stabilising $v$ has to be the scalar action on each $r_k$-dimensional space $\langle v_i\rangle_{i\in I_k}$. Hence, the dimension of each orbit $\pi^{-1}(\varphi[\ast])$ is denoted by the number before the semicolon. 

\begin{thm} \label{thm:single} For the map $\varpi\colon (\mathbb P^2)^4\to \mathcal P_4$ defined in Definition \ref{df:splitting}, each fibre of $\varpi$ at all elements in $\mathcal P_{3,4}'=\mathrm{Image}(\varpi)\setminus\{\{([4],2)\}\}$ is a single $G$-orbit coincides with the fibre $\pi^{-1}(\varphi[\ast])$ listed in Table \ref{tab:vpsingle}. Furthermore, representatives of these orbits are also listed in Table \ref{tab:vpsingle}.  
\end{thm}

\subsection{Description of infinitely many orbits}
\label{orbdecomp}

In Section \ref{prbdecomp}, we described the decomposition of the multiple projective space $(G/Q)^4\cong(\pr{2})^4$ into a finite number of $G$-invariant fibres of $\varpi\colon (\mathbb P^2)^4\to\mathcal P_4$ defined in Definition \ref{df:splitting}. Furthermore, for the subset $\mathcal P_{3,4}'\subset \mathrm{Image}(\varpi)$ defined in Lemma \ref{thm:rhobij}, each fibre of $\varpi$ on it was single $G$-orbit (see Theorem \ref{thm:single}). On the other hand, the fibre of $\{([4],2)\}$, which is the only element in $\mathrm{Image}(\varpi)\setminus \mathcal P_{3,4}'$, is not. In this section, 
we describe the $G$-orbit decomposition of this extra fibre. For this aim, we define some notations as follows: 
\bedef \label{df:mapsplit}
\been
\item The subset $(\pr{1})'$ of $\pr{1}$ denotes 
$
\pr{1}\setminus\left\{[e_1],[e_2],[e_1+e_2] \right\}. 
$
\item Define the map $\varphi[6],\varphi[5;i,j]\in Map(2^{[4]},\mathbb N)$ for $1\leq i<j\leq 4$ as follows: 
\[\varphi[6]\colon I\mapsto \min\{\#I,2\},\ \ \ \varphi[5;i,j]\colon I\mapsto \min\{\#(I/i\sim j),2\}.\]
\item For $p=[p_1e_1+p_2e_2]\in\pr{1}$, we define a subset of $(\pr{2})^4$ by 
\[
\orb{5;p}:=\left\{\left.[v_i]_{i=1}^4\in (\pr{2})^4\right|[v_1]\neq [v_2],\;v_3=v_1+v_2,\;v_4=p_1v_1+p_2v_2\right\}.
\]
\enen
\end{df}

Consider an element $v=[v_i]_{i=1}^4\in (\mathbb P^2)^4$, then it lies in the fibre of $\{([4],2)\}$ if and only if $\{v_i\}_{i=1}^4$ spans $2$-dimensional space and is indecomposable. 
If $\dim\langle v_i\rangle_{i=1}^4=2$, there exists a pair $\{v_i,v_j\}$ which forms a basis of this space. Then $\{v_i,v_j,v_k,v_l\}$ is indecomposable if and only if either $v_k$ or $v_l$ is contained in neither $\langle v_i\rangle$ nor $\langle v_j\rangle$ where $\{i,j,k,l\}=[4]$. Hence, $\varpi(v)=\{([4],2)\}$ if and only if its rank is $2$ and there exists a triple of $\{[v_i]\}_{i=1}^4$ which are distinct. 

Let $\{[v_1],[v_2],[v_3]\}$ be a distinct triple, then there is a isomorphism $g\in GL_3$ sending it to $\{[e_1],[e_2],[e_1+e_2]\}$. Then the line $[v_4]$ is sent to some $p=[p_1e_1+p_2e_2]\in\mathbb P^1$. Remark that $p$ is independent of the choice of $G$-action since $G$-action stabilising $[e_1,e_2,e_1+e_2]$ has to be a scalar action on this $2$-dimensional space. 

Now if $p\in \left(\mathbb P^1\right)'$, then we have $\pi(v)=\varphi[6]$. On the other hand, $\pi(v)=\varphi[5;1,4]$, $\varphi[5;2,4]$, or $\varphi[5;3,4]$ if $p=[e_1]$ , $[e_2]$, or $[e_1+e_2]$ with respectively. By considering other distinct triples, all $\varphi[5;\ast]$ can be covered, and we obtain the following: 

\begin{thm} \label{thm:parameter} For the maps $\varpi\colon (\pr{2})^4\overset{\pi}{\to} Map(2^{[4]},\mathbb N)\overset{\rho}{\to}\mathcal P_{4}$ defined in Definition \ref{df:splitting}, (\ref{eq:rankmatrix}), and Lemma \ref{thm:rho}, 
\begin{enumerate}
\item for the only element $\{([4],2)\}$ in $\mathrm{Image}(\varpi)\setminus\mathcal P_{3,4}'$, the fibre of $\left.\rho\right|_{\mathrm{Image}(\pi)}$ is fullfilled with the maps $\varphi[6]$ and $\varphi[5;i,j]$ ($1\leq i<j\leq 4$) defined in Definition \ref{df:mapsplit}. 
\item Each fibre $\pi^{-1}(\varphi[5;i,j])$ is a single $G$-orbit through $[e_1,e_1,e_2,e_1+e_2]$ (if $\{i,j\}=\{1,2\}$), and the fibre $\pi^{-1}(\varphi[6])$ is decomposed into $G$-orbits as follows: 
\[\pi^{-1}(\varphi[6])=\coprod_{p\in\left(\mathbb P^1\right)'}\mathcal O(5;p)=\coprod_{p\in\left(\mathbb P^1\right)'}G\cdot[e_1,e_2,e_1+e_2,p].\]
\end{enumerate}
\end{thm}

\section{Closure relations among orbits}

In this section, we determine the closure relations among $GL_3$-orbits on $(\pr{2})^4$ and $GL_3$-invariant fibres described in Theorems \ref{thm:single} and \ref{thm:parameter}. 

For an integer $d$, consider the $\mathbb K$-submodule $\mathcal H^d_n$ of the $n$-variable polynomial ring consisting of all homogeneous polynomials of the degree $d$. Then the polynomial ring $Pol(M(n,m))$ over $n\times m$-matrices admits a $\mathbb N^m$-graded structure as $Pol(M(n,m))=\bigoplus_{\bm d\in \mathbb N^m}\mathcal H_{n,m}^{\bm d}$ where $\mathcal H_{n,m}^{(d_1,d_2,\ldots,d_m)}:=\mathcal H_n^{d_1}\otimes \mathcal H^{d_2}_n\otimes\cdots\otimes \mathcal H^{d_m}_n$. 
Then closed subsets in $(\mathbb P^{n-1})^m$ correspond to the zero-point sets $Z(I)$ of homogeneous ideals $I=\bigoplus I\cap\mathcal H_{n,m}^{\bm d}$ of  $Pol(M(n,m))$. 
In particular, the irreducibility of $Z(I)$ also corresponds to whether $I$ is primary or not.

\bedef \label{thm:varphirelation}
For $\varphi$, $\psi\in Map(2^{[m]},\mathbb  N)$, we say that $\varphi\leq\psi$ if $\varphi(I)\leq\psi(I)$ for all $I\subset[m]$. 
\end{df}
With this notation,  we state the result on the closure relations among $\GL{3}$-invariant fibres of  the map $\pi\colon (\mathbb P^2)^4\to Map(2^{[4]},\mathbb N)$ defined in (\ref{eq:rankmatrix}). 
\beprop \label{thm:clsingle}
For $\varphi\in \mathrm{Image}(\pi)\subset Map(2^{[4]},\mathbb N)$, we have 
$\overline{\pi^{-1}(\varphi)}=\coprod_{\psi\leq\varphi}\pi^{-1}(\psi)$. 
\enprop
\begin{proof}
For $\varphi\in Map(2^{[m]},\mathbb N)$, we have 
\begin{align*}
\coprod_{\psi\leq\varphi}\pi^{-1}(\psi)&=\left\{[v]\in(\pr{n-1})^m\left|\ \dim\langle v_i\rangle_{i\in I}\leq \varphi(I) \textrm{ for all }I\subset [m]\right.\right\}  \\
&=\left\{[v]\in(\pr{n-1})^m\left|\ \textrm{ all }(\varphi(I)+1)\textrm{-minors of }(v_i)_{i\in I}\textrm{ vanish for all }I\subset[m]\right.\right\}, \\
\pi^{-1}(\varphi)&=\left\{[v]\in(\pr{n-1})^m\left|\ \dim\langle v_i\rangle_{i\in I}= \varphi(I) \textrm{ for all }I\subset [m]\right.\right\}  \\
&=\left\{[v]\in(\mathbb P^{n-1})^m\left|\begin{array}{l}\textrm{there exists some non vanishing }\\ \varphi(I)\textrm{-minor of }(v_i)_{i\in I} \textrm{ for all }I\subset[m]\end{array}\right.\right\}\ \cap\ \coprod_{\psi\leq\varphi}\pi^{-1}(\psi).
\end{align*}
From this characterisation, $\coprod_{\psi\leq \varphi}\pi^{-1}(\psi)$ is a closed subset defined by the homogeneous ideal generated by all $(\varphi(I)+1)$-minors consisting of columns contained in $I$, and $\pi^{-1}(\varphi)$ is open in it. Hence we have $\overline{\pi^{-1}(\varphi)}\subset \coprod_{\psi\leq \varphi}\pi^{-1}(\psi)$. 
Furthermore, since $\pi^{-1}(\varphi)$ is not empty for $\varphi\in\mathrm{Image}(\pi)$, we can prove the converse inclusion by checking the irreducibility of $\coprod_{\psi\leq\varphi}\pi^{-1}(\psi)$. 

Now, let $X_{n,m}(r)\subset (\mathbb P^{n-1})^m$ denote the irreducible closed subset consisting of all matrices whose rank is less than or equal to $r$. Then for the maps $\varphi=\varphi[8]$, $\varphi[7;i]$, $\varphi[6;i,j]$, $\varphi[6]$, $\varphi[5;i,j]$, $\varphi[4;\ast]$, $\varphi[2]$ in $\mathrm{Image}(\pi)$ introduced in Definitions \ref{df:map} and \ref{df:mapsplit}, the closed subsets $\coprod_{\psi\leq \varphi}\pi^{-1}(\psi)$ is naturally identified with $X_{3,4}(3)$, $X_{3,3}(2)\times X_{3,1}(1)$, $X_{3,3}(3)$, $X_{3,4}(2)$, $X_{3,3}(2)$, $X_{3,2}(2)$, and $X_{3,1}(1)$ respectively. Hence they are all irreducible and the claim holds. 
\end{proof}
According to Theorems \ref{thm:single} and \ref{thm:parameter}, the $GL_3$-invariant fibres $\pi^{-1}(\varphi)$ are all $GL_3$-orbits by themselves unless $\varphi=\varphi[6]$. On the other hand, $\pi^{-1}(\varphi[6])$ is decomposed in infinitely many orbits $\orb{5;p}$ introduced in Definition \ref{df:mapsplit}. Hence in the next, we shall determine the closure of $\orb{5;p}$, which completes the determination of all closure relations among orbits.  
\beprop \label{thm:clparameter}
For the orbits $\orb{5;p}:=GL_3\cdot([e_1], [e_2], [e_1+e_2],p)$, $\pi^{-1}(\varphi[4;i])$, and $\pi^{-1}(\varphi[2])$ where $p\in(\pr{1})'$ and $1\leq i\leq 4$ (see Definitions \ref{df:map} and \ref{df:mapsplit}) , we have 
\[
\overline{\orb{5;p}}=\orb{5;p}\ \amalg\ \coprod_{i=1}^4\pi^{-1}(\varphi[4;i])\ \amalg\ \pi^{-1}(\varphi[2]).
\] 
\enprop 
\belem \label{thm:iotahomeo}For an linear inclusion $\iota\colon \mathbb K^r\hookrightarrow\mathbb K^n$, the map below is a closed embedding.  
\[\tilde\iota\colon GL_r\backslash(\mathbb P^{r-1})^m \to GL_n\backslash (\mathbb P^{n-1})^m\colon [v_i]_{i=1}^m \mapsto [\iota(v_i)]_{i=1}^m. \]
\enlem
\begin{proof}
Since $g\in GL_r$ defines a linear isomorphism in $\mathrm{Image}(\iota)$, there exists a linear isomorphism $\tilde g\in GL_n$ such that $\iota\circ g=\tilde g\circ \iota$ by taking the direct product with some linear isomorphism on the complement space. Hence the map $\tilde \iota$ is well-defined. On the other hand, for a linear isomorphism $\tilde g\in GL_n$ which satisfies $\tilde g\cdot[\iota(v_i)]_{i=1}^m=[\iota(w_i)]_{i=1}^m$, it defines a linear isomorphism between $\langle v_i\rangle_{i=1}^m$ and $\langle w_i\rangle_{i=1}^m$ sending $[v_i]$ to $[w_i]$ via $\iota$. So there is an isomorphism $g\in GL_r$ satisfying $g[v_i]=[w_i]$. Hence the map $\tilde\iota$ is injective. 

Since the continuity is obvious, we shall only show that the map is closed. For a closed subset $C$ of $GL_r\backslash (\mathbb P^{r-1})^m$, there exists a $GL_r$-invariant zero point set $Z(I)\subset (\mathbb P^{r-1})^m$ such that $C=GL_r\backslash Z(I)$ where $I$ is a homogeneous ideal in $Pol(M(r,m))$. Then we have $\tilde\iota(C)=GL_n\backslash GL_n\cdot\iota(Z(I))$. 

Now we define a homogeneous ideal $J\subset Pol(M(n,m))$ generated by homogeneous polynomials $f\circ p$ for all $f\in I$ and linear maps $p\colon \mathbb K^n\to\mathbb K^r$. 

Let $v=[v_i]_{i=1}^m \in Z(J)\cap X_{n,m}(r)$ where $X_{n,m}(r)$ is a closed subset of $(\mathbb P^{n-1})^m$ consisting of all elements with the rank less than $r+1$. Then there exixts a linear isomorphism $g\in GL_n$ sending $v_i$ into the $r$-dimensional space $\mathrm{Image}(\iota)$. Hence, taking some linear projection $p\colon \mathbb K^n\twoheadrightarrow\mathbb K^r$ such that  $\iota p$ is the identity map on $\mathrm{Image}(\iota)$, we have $v=g^{-1}gv=g^{-1}\iota pgv\in GL_n\cdot \iota(pgv)$, and $f(pgv)=f\circ(pg)(v)=0$ for all $f\in I$. Hence $v\in GL_n\cdot \iota(Z(I))$. 

On the other hand, let $v \in Z(I)$. Since $Z(I)$ is $GL_r$-invariant, $f(gv)=0$ for all $g\in GL_r$. Hence $f(xv)=0$ holds for all linear endomorphisms $x$ on $\mathbb K^r$, which leads that $f\circ p(\iota(v))=f(p\iota v)=0$ for all $p\colon \mathbb K^n\to\mathbb K^r$. Hence $\iota(Z(I))\subset Z(J)$. 

From these arguments, we have shown that $GL_n\cdot \iota(Z(I))=Z(J)\cap X_{n,m}(r)$, which leads that the map $\tilde\iota$ is closed. 

\end{proof}
From this lemma, we only have to determine the closure of 
$GL_2\cdot [e_1,e_2,e_1+e_2,p]=\tilde\iota^{-1}(\mathcal O(5;p))\subset(\mathbb P^1)^4$. Remark that $\tilde\iota$ intertwines the map $\pi$. For this aim, we introduce the following notation: for a $2\times 4$-matrix, $|i,j|$ denotes the $2$-minor with the $i$ and $j$-th columns. Then we define a homogeneous polynomial $P_{p}\in \mathcal{H}^{(1,1,1,1)}_{2,4}$ for $p=[p_1e_1+p_2e_2]\in\mathbb P^1$ by
\begin{align} \label{eq:poly}
P_{p}(x):=&p_2|1,3||4,2|+p_1|1,4||2,3| \\
=&-p_1|1,2||3,4|+(p_2-p_1)|1,3||4,2|=-p_2|1,2||3,4|+(p_1-p_2)|1,4||2,3|. \notag
\end{align}
Remark that the equalities hold since $|1,2||3,4|+|1,3||4,2|+|1,4||2,3|=0$ from the general property of determinants. Considering the $\mathcal S_4$-action on the columns, we have 
\[
P_{p}((1,2)x)=P_{{}^t(-p_2,-p_1)}(x), \;\; P_{p}((2,3)x)= P_{{}^t(p_2-p_1,p_2)}(x),\;\;
P_{p}((3,4)x)=P_{{}^t(-p_2,-p_1)}(x).
\]
Hence, for the group homomorphism $\Psi\colon\mathcal{S}_4\rightarrow\GL{2}$ generated by $(1,2), (3,4)\mapsto\scalebox{0.5}{$\left(\bear{cc}0&-1\\-1&0 \enar\right)$}$ and $(2,3)\mapsto\scalebox{0.5}{$\left(\bear{cc}-1&1\\0&1\enar\right)$}$, we have
\begin{align}
P_{p}(\sigma^{-1} x)=P_{\Psi(\sigma)p}(x)\;\;\textrm{for }\sigma\in\mathcal{S}_4. 
\end{align}
Remark that the kernel of $\Psi$ is the Klein $4$-group $\{\mathrm{id}, (1,2)(3,4), (1,3)(2,4), (1,4)(2,3)\}$, and the image of  $\Psi$ acts on the $3$ points $\{0,1,\infty\}\subset\pr{1}$ effectively, which is isomorphic to $\mathcal{S}_3$. 
\belem \label{thm:polyirred}
The polynomial $P_{p}$ is irreducible if and only if $p\in(\pr{1})'$ (see Definition \ref{df:mapsplit}).  
\enlem
\begin{proof}
Remark that $p$ is in $\left(\mathbb P^1\right)'$ if and only if $p_1p_2(p_1-p_2)\neq 0$. Since the polynomial $P_p$ is of the homogeneous degree $(1,1,1,1)$, we only have to check that it does not have any divisors with the homogeneous degree neither $(1,0,0,0)$, $(0,1,0,0)$, $(0,0,1,0)$, $(0,0,0,1)$, $(1,1,0,0)$, $(1,0,1,0)$, nor $(1,0,0,1)$.



By an direct computation, we have  
\begin{align*}
P_{}(x)&=x_{1,1}x_{1,2}((p_1-p_2)x_{2,3}x_{2,4})+x_{2,1}x_{2,2}((p_1-p_2)x_{1,3}x_{1,4}) \\
&\;\;+x_{1,1}x_{2,2}(p_1x_{1,3}x_{2,4}-p_2x_{2,3}x_{1,4})-x_{2,1}x_{1,2}(p_2x_{1,3}x_{2,4}-p_1x_{2,3}x_{1,4}). 
\end{align*}
Hence it is divided by a polynomial of the homogeneous degree ${(1,1,0,0)}$ if and only if $p_1= p_2$. 
Similarly, since $P_{p}((2,3)x)=P_{{}^t(p_2-p_1,p_2)}(x)$ and $P_{p}((2,4)x)=P_{{}^t(p_1,p_1-p_2)}(x)$, 
$P_{p}(x)$ is divided by a polynomial of the homogeneous degree ${(1,0,1,0)}$ (resp. ${(1,0,0,1)}$) if and inly if $p_1=0$ (resp. $p_2=0$). 

On the other hand, let $P_{p}(x)$ has a divisor $f(x_1)$ of the homogeneous degree $(1,0,0,0)$. 
Since $P_{p}(x)$ is fixed under the action of $(1,2)(3,4)$, $(1,3)(2,4)$, and $(1,4)(2,3)$, it is factorised as $P_{\tilde{p}}(x)=cf(x_1)f(x_2)f(x_3)f(x_4)$, it can occur only if $p_1p_2(p_1-p_2)=0$ from the argument above. 
\end{proof}

Next, we observe the relationship between the polynomial $P_{p}$ introduced in (\ref{eq:poly}) and the orbit $GL_2\cdot[e_1,e_2,e_+e_2,p]$:
\belem \label{thm:orb5pcharacterise}
For $v_p=[e_1,e_2,e_1+e_2,p]\in(\mathbb P^1)^4$ and the homogeneous polynomial $P_{p}$ introduced in (\ref{eq:poly}), we have the following: 
\bealign
{\GL{2}}\cdot v_p=\left\{[v]\in(\pr{1})^4\left|\ |1,2||2,3||3,1|\neq 0,
\;P_{p}(v)=0\right.\right\}. \label{eq:orb5pcharacterise}
\end{align}
\enlem
\begin{proof}
First of all, $P_{p}(x)$ is relatively $GL_2$-invariant since all minors are relatively $GL_2$-invariant with the character $g\mapsto (\det g)^2$. Let $X_p\subset(\pr{1})^4$ be the right-hand side of (\ref{eq:orb5pcharacterise}), then the orbit $GL_2\cdot v_p$ is included in $X_p$ clearly from $v_p\in X_p$ and the $GL_2$-invariance of $P_p$. 

Conversely, for $[v]=[v_1,v_2,v_3,v_4]\in X_p$, then $\{v_1,v_2\}$ is linearly independent and 
\begin{align*}
|2,3|v_1+|3,1|v_2&=-|1,2|v_3\\
|1,4|\left(p_1|2,3|v_1+p_2|3,1|v_2\right)&=p_1|1,4||2,3|v_1+p_2|1,4||3,1|v_2 \\
&=-p_2|1,3||4,2|v_1+p_2|1,4||3,1|v_2=-p_2|3,1|\left(|2,4|v_1+|4,1|v_2\right)\\
&=p_2|3,1||1,2|v_4
\end{align*}
Hence $v=(|2,3|v_1\ |3,1|v_2)\cdot [e_1,e_2,e_1+e_2,p]\in GL_2\cdot v_p$. 
\end{proof}

From Lemmas \ref{thm:polyirred} and \ref{thm:orb5pcharacterise}, the orbit $GL_2\cdot v_p$ is open in the irreducible closed subset $Z(P_p)$ where $p\in (\mathbb P^1)'$. Hence we have $\overline{GL_2\cdot v_p}=Z(P_p)$, and shall only determine the set $Z(P_p)$. 
\belem \label{thm:polysolution}
Let $p\in(\pr{1})'$ and $v_p:=[e_1, e_2, e_1+e_2, p]\in \mathcal O_p$. 
Then the zero point set of the polynomial $P_p$ defined in (\ref{eq:poly}) is given as follows: 
\beeq \label{eq:polysolution}
Z(P_p) = GL_2\cdot v_p \amalg\coprod_{i=1}^4\pi^{-1}(\vp{4;i})\amalg\pi^{-1}(\vp{2}).
\eneq
\enlem
\begin{proof}
We already have classified orbits in $GL_2\backslash (\mathbb P^1)^4\hookrightarrow GL_3\backslash (\mathbb P^2)^4$ in Theorems \ref{thm:single} and \ref{thm:parameter}. 
Hence, we shall only determine whether each of them is contained in $Z(P_p)$ or not. 

If any $2$-minors do not vanish, we only have to consider the orbit $GL_2\cdot v_q= \tilde\iota^{-1}(\mathcal O(5;q))$ where $q\in(\pr{1})'$. Then since $P_p(v_q)=p_2q_1-p_1q_2$, it is contained in $Z(P_p)$ if and only if $p=q$. 

On the other hand, consider the case where exists some vanishing $2$-minor. Since the polynomial $P_p$ is of the form $c_1|i,j||k,l|+c_2|i,k||l,j|$ for some $c_1c_2(c_1-c_2)\neq 0$, if $P_p(v)=0$, then there has to be a triple of columns of $v$ whose $2$-minors all vanish. Conversely, if there exists a triple of columns whose $2$-minors all vanish, then clearly $P_p(v)=0$. 
\end{proof}

Combining these results, we completed the proof of the closure relations in Theorem \ref{thm:orb34}. 

\end{document}